\documentclass{article}
\usepackage{amssymb}
\usepackage{amsmath}
\usepackage{amsthm}
\usepackage{mathabx}
\usepackage{enumerate}
\usepackage{xcolor}

\usepackage{hyperref}

\theoremstyle{plain}

\newtheorem{Cor}{Corollary}

\newtheorem{Thm}{Theorem}

\newcommand*{\Ann}{\ensuremath{\mathrm{Ann\,}}}
\newcommand*{\R}{\ensuremath{\mathbb{R}}}

\newcommand*{\Z}{\ensuremath{\mathbb{Z}}}
\newcommand*{\C}{\ensuremath{\mathbb{C}}}

\begin{document}
	
	\date{}
	
	\author{
		L\'aszl\'o Sz\'ekelyhidi\\
		{\small\it Institute of Mathematics, University of Debrecen,}\\
		{\small\rm e-mail: \tt szekely@science.unideb.hu,}
	}
	
	\title{New Results on Spectral Synthesis}

	\maketitle
	
	\begin{abstract}
		In our former paper we introduced the concept of localization of ideals in the Fourier algebra of a locally compact Abelian group. It turns out that localizability of a closed ideal in the Fourier algebra is equivalent to the synthesizability of the annihilator of that closed ideal which corresponds to this ideal in the measure algebra. This equivalence provides an effective tool to prove synthesizability of varieties on locally compact Abelian groups. In this paper we utilize this tool to show that when investigating synthesizability of a variety, roughly speaking compact elements of the group can be neglected. Our main result is that spectral synthesis holds on a locally compact Abelian group $G$ if and only if it holds on $G/B$, where $B$ is the closed subgroup of all compact elements. In particular, spectral synthesis holds on compact Abelian groups. Also we obtain a simple proof for the characterization theorem of spectral synthesis on discrete Abelian groups.
	\end{abstract}

	\footnotetext[1]{The research was supported by the  the
	Hungarian National Foundation for Scientific Research (OTKA),
	Grant No.\ K-134191.}\footnotetext[3]{AMS (2000) Subject Classification: 43A45, 22D99}
	
	\section{Introduction}
	Let $G$ be a locally compact Abelian group. Spectral synthesis deals with uniformly closed translation invariant linear spaces of continuous complex valued functions on $G$. Such a space is called a {\it variety}. We say that {\it spectral analysis} holds for a variety, if every nonzero subvariety contains a one dimensional subvariety. We say that a variety is {\it synthesizable}, if its finite dimensional subvarieties span a dense subspace in the variety. Finally, we say that {\it spectral synthesis} holds for a variety, if every subvariety of it is synthesizable. On commutative topological groups finite dimensional varieties of continuous functions are completely characterized: they are spanned by exponential monomials. {\it Exponential polynomials} on a topological Abelien group are defined as the elements of the complex algebra of continuous complex valued functions generated by all continuous homomorphisms into the multiplicative group of nonzero complex numbers ({\it exponentials}), and all continuous homomorphisms into the additive group of all complex numbers ({\it additive functions}).  An {\it exponential monomial} is a function of the form
	$$
	x\mapsto P\big(a_1(x),a_2(x),\dots,a_n(x)\big)m(x),
	$$
	where $P$ is a complex polynomial in $n$ variables, the $a_i$'s are additive functions, and $m$ is an exponential. Every exponential polynomial is a linear combination of exponential monomials. For more about spectral analysis and synthesis on groups see \cite{MR2680008,MR3185617}.
	\vskip.2cm
	
	In \cite{MR2340978}, the authors characterized those discrete Abelian groups having spectral synthesis: spectral synthesis holds for every variety on the discrete Abelian group $G$, if and only if $G$ has finite torsion free rank. In particular, from this result it follows, that if spectral synthesis holds on $G$ and $H$, then it holds on $G\oplus H$. Unfortunately, such a result does not hold in the non-discrete case. Namely, by the fundamental result of L.~Schwartz \cite{MR0023948}, spectral synthesis holds on $\R$, but D.~I.~Gurevich showed in \cite{MR0390759} that spectral synthesis fails to hold on $\R\times\R$. In this paper we prove that spectral synthesis holds on the locally compact Abelian group $G$ assuming that it holds on the factor group $G/B$, where $B$ is the closed subgroup of $G$ formed by all compact elements. We recall, that in a topological group an element is called {\it compact element}, if the closure of the cyclic subgroup generated by the element is compact. In a locally compact Abelian group all compact elements form a closed subgroup (see \cite[(9.10) Theorem]{MR0156915}).
	
\section{Localization}
\indent In our former paper \cite{Sze23} we introduced the concept of localization of ideals in the Fourier algebra of a locally compact Abelian group. We recall this concept here.
\vskip.2cm
Let $G$ be a locally compact Abelian group and let $\mathcal A(G)$ denote its Fourier algebra, that is, the algebra of all Fourier transforms of compactly supported complex Borel measures on $G$. This algebra is topologically isomorphic to the measure algebra $\mathcal M_c(G)$. For the sake of simplicity, if the annihilator $\Ann I$ of the closed ideal $I$ in $\mathcal M_c(G)$ is synthesizable, then we say that the corresponding closed ideal $\widehat{I}$ in $\mathcal A(G)$ is synthesizable. For each derivation $D$ on $\mathcal A(G)$, we introduced (see \cite{Sze23}) the set $\widehat{I}_{D,m}$ as the set of all functions $\widehat{\mu}$ in $\mathcal A(G)$ for which
$$
D\widehat{\mu}(m)=\int \Delta_{x,y_1,y_2,\dots,y_k}*f_{D,m}(0)\widecheck{m}(x)\,d\mu(x)=\widehat{\mu}(m)=0
$$
holds for each $k=1,2,\dots$ and $y_1,y_2,\dots,y_k$ in $G$. Then $\widehat{I}_{D,m}$ is a closed ideal in $\mathcal A(G)$. For a family $\mathcal D$ of derivations we write
$$
\widehat{I}_{\mathcal{D},m}=\bigcap_{D\in \mathcal D} \widehat{I}_{D,m}.
$$
Clearly, $\widehat{I}_{\mathcal{D},m}$ is a closed ideal as well. 

The dual concept is the following: given an ideal $\widehat{I}$ in $\mathcal A(G)$ and an exponential $m$, the set of all derivations on $\mathcal A(G)$ which annihilate $\widehat{I}$ at $m$ is denoted by $\mathcal{D}_{\widehat{I},m}$. The subset of $\mathcal{D}_{\widehat{I},m}$ consisting of all polynomial derivations is denoted by $\mathcal{P}_{\widehat{I},m}$. We have the basic inclusion
\begin{equation}\label{basic}
\widehat{I}\subseteq \bigcap_m \widehat{I}_{\mathcal{D}_{\widehat{I},m},m}\subseteq \bigcap_m \widehat{I}_{\mathcal{P}_{\widehat{I},m},m}.
\end{equation}
We note that if $m$ is not a root of $\widehat{I}$, then $\mathcal{D}_{\widehat{I},m}=\mathcal{P}_{\widehat{I},m}=\{0\}$, consequently $\widehat{I}_{\mathcal{D}_{\widehat{I},m},m}=\widehat{I}_{\mathcal{P}_{\widehat{I},m},m}=\mathcal A(G)$, hence those terms have no effect on the intersection. 
\vskip.2cm

The ideal $\widehat{I}$ is called {\it localizable}, if we have equalities in \eqref{basic}. Roughly speaking, localizability of an ideal means, that the ideal is completely determined by the values of "derivatives" of the functions belonging to this ideal. The main result in \cite{Sze23} is that $\widehat{I}$ is synthesizable if and only if it is localizable. We shall use this result in the subsequent paragraphs.

\section{The main result}

Our main result in this paper is the following.

\begin{Thm}\label{main}
	Let $G$ be a locally compact Abelian group and let $B$ denote the closed subgroup of $G$ consisting of all compact elements. Then spectral synthesis holds on $G$ if and only if it holds on $G/B$.
\end{Thm}

\begin{proof}
	If spectral synthesis holds on $G$, then it obviously holds on every continuous homomorphic image of $G$, in particular, it holds on $G/B$.
	\vskip.2cm
	
	Conversely,  we assume that spectral synthesis holds on $G/B$. This means, that every closed ideal in the Fourier algebra of $G/B$ is localizable, and we need to show the same for all closed ideals of the Fourier algebra of $G$. 
	\vskip.2cm
	
	First we remark that the polynomial rings over $G$ and over $G/B$ can be identified. Indeed, polynomials on $G$ are built up from real characters on $G$, which clearly vanish on compact elements, as the topological group of the reals has no nontrivial compact subgroups. Consequently, if $a$ is a real character, and $x,y$ are in the same coset of $B$, then $x-y$ is in $B$, and $a(x-y)=0$, which means $a(x)=a(y)$. So, the real characters on $G$ arise from the real characters of $G/B$, hence the two polynomial rings can be identified. 
	\vskip.2cm
	
	Now we define a projection of the Fourier algebra of $G$ into the Fourier algebra of $G/B$ as follows. Let $\Phi:G\to G/B$ denote the natural mapping. For each measure $\mu$ in $\mathcal M_c(G)$ we define $\mu_B$ as the linear functional 
	$$
	\langle \mu_B,\varphi\rangle=\langle \mu, \varphi\circ \Phi\rangle
	$$
	whenever $\varphi:G/B\to\C$ is a continuous function. It is straightforward that the mapping $\widehat{\mu}\mapsto \widehat{\mu}_B$ is a continuous algebra homomorphism of the Fourier algebra of $G$ into the Fourier algebra of $G/B$. As $\Phi$ is an open mapping, closed ideals are mapped onto closed ideals. 
	\vskip.2cm
	
	For a given closed ideal $\widehat{I}$ in $\mathcal A(G)$, we denote by $\widehat{I}_B$ the closed ideal in $\mathcal A(G/B)$ which corresponds to $\widehat{I}$ under the above homomorphism. If $m$ is a root of the ideal $\widehat{I}_B$, then $\widehat{\mu}_B(m)=0$ for each $\widehat{\mu}$ in $\widehat{I}$. In other words,
	$$
\langle\mu,\widecheck{m}\circ \Phi\rangle= \langle{\mu}_B,\widecheck{m}\rangle=0,
	$$
	hence $m\circ \Phi$, which is clearly an exponential on $G$, is a root of $\widehat{I}$. Suppose that $D$ is a derivation in $\mathcal P_{\widehat{I},m\circ \Phi}$, then it has the form
	$$
	D\widehat{\mu}(m\circ \Phi)=\int p\cdot (\widecheck{m}\circ \Phi)\,d\mu
	$$
	with some polynomial $p$ on $G$ (see \cite{Sze23}). According to our remark above, the polynomial $p$ can uniquely be written  as $p_B\circ \Phi$, where $p_B$ is a polynomial on $G/B$. In other words,
	$$
	D\widehat{\mu}(m\circ \Phi)=\langle \widehat{\mu},(p_B\circ \Phi)(\widecheck{m}\circ \Phi)\rangle =\langle \widehat{\mu}_B,  p_B\widecheck{m}\rangle=D_B(\widehat{\mu}_B)(m),
	$$
	which defines a derivation $D_B$ on $\mathcal A(G/B)$ with generating function $f_{D_B,m}=p_B$ (see \cite{Sze23}).
	\vskip.2cm
	
	It follows, that every derivation in $\mathcal P_{\widehat{I},m\circ \Phi}$ arises from a derivation  in $\mathcal P_{\widehat{I}_B,m}$. On the other hand, if $d$ is a derivation in $\mathcal P_{\widehat{I}_B,m}$, then we have
	$$
	d\widehat{\mu}_B(m)=\int p \widecheck{m}\,d\mu_B=\langle \mu_B,p \widecheck{m}\rangle=\langle \mu, (p\circ \Phi) (\widecheck{m}\circ \Phi)\rangle=
	$$
	$$
	\int p(\Phi(x) (\widecheck{m}\circ \Phi)(x)\,d\mu(x),
	$$
	which defines a derivation $D$ in $\mathcal P_{\widehat{I},m\circ \Phi}$.
	\vskip.2cm
	
	We summarize our assertions. Let $\widehat{I}$ be a proper closed ideal in $\mathcal A(G)$ and assume that $\widehat{I}$ is non-localizable. It follows that there is a function $\widehat{\nu}$ not in $\widehat{I}$ which is annihilated at $M$ by all polynomial derivations in $\mathcal P_{\widehat{I},M}$, for each exponential $M$ on $G$. In particular,  $\widehat{\nu}$ is annihilated at $m\circ \Phi$ by all polynomial derivations in $\mathcal P_{\widehat{I},m\circ \Phi}$, for each exponential $m$ on $G/B$. We have seen above that this implies that $\widehat{\nu}_B$ is annihilated at $m$ by all polynomial derivations in $\mathcal P_{\widehat{I}_B,m}$ and for each exponential $m$ on $G/B$. As spectral synthesis holds on $G/B$, the ideal $\widehat{I}_B$ is localizable, hence $\widehat{\nu}_B$ is in $\widehat{I}_B$, but this contradicts the assumption that $\widehat{\nu}$ is not in $\widehat{I}$. The proof is complete.
\end{proof}

\begin{Cor}\label{compelem}
	If every element of a locally compact Abelian group is compact, then spectral synthesis holds on this group.
\end{Cor}

\begin{proof}
	If every element of $G$ is compact, then $G=B$, and $G/B$ is a finite group.
\end{proof}

\begin{Cor}\label{comp}
Spectral synthesis holds on every compact Abelian group.
\end{Cor}

\begin{proof}
Again, we have $G=B$.
\end{proof}
	
\begin{Cor}\label{tors}
Spectral synthesis holds on a discrete Abelian group if and only if its torsion free rank is finite (see \cite{MR2340978, Sze23}).
\end{Cor}

\begin{proof}
	If the torsion free rank of $G$ is infinite, then there is a non-polynomial derivation on its Fourier algebra, hence we have the chain of inclusions
	$$
	\widehat{I}\subseteq \widehat{I}_{\mathcal D_{{\widehat{I},m}},m}\subsetneq \widehat{I}_{\mathcal P_{{\widehat{I},m}},m},
	$$
	which implies that $\widehat{I}\ne \widehat{I}_{\mathcal P_{{\widehat{I},m}},m}$, hence $\widehat{I}$ is not synthesizable.
	\vskip.2cm
	
	Conversely, let $G$ have finite torsion free rank. The subgroup $B$ of compact elements coincides with the set $T$ of all elements of finite order, and $G/T$ is a (continuous) homomorphic image of $\Z^n$ with some nonnegative integer $n$. As spectral synthesis holds on $\Z^n$, it holds on its homomorphic images. 
\end{proof}

\section{Data availability statement}

Data sharing not applicable to this article as no datasets were generated or analysed during the current study.

	\end{document}